\documentclass[a4paper]{article}

\usepackage{amsmath}
\usepackage{amssymb}
\usepackage[normalem]{ulem}

\usepackage{ntheorem}

\newtheorem{lemma}[subsection]{Lemma}

\newtheorem{theorem}[subsection]{Theorem}
\newtheorem*{theorem-nonum}{Theorem}
\newtheorem{corollary}[subsection]{Corollary}
\newtheorem{claim}[subsection]{Claim}

\newtheorem{definition}[subsection]{Definition}
   {\refstepcounter{subsection}%
        \medbreak\noindent{\bf Question \thesubsection\space}}%
   {\par\medbreak}%
   {\refstepcounter{subsection}%
        \medbreak\noindent{\bf Conjecture \thesubsection\space}}%
   {\par\medbreak}%
\newenvironment{remark}%
   {\refstepcounter{subsection}%
        \medbreak\noindent{\bf Remark \thesubsection\space}}%
   {\par\medbreak}%
   {\refstepcounter{subsection}%
        \medbreak\noindent{\bf Example \thesubsection\space}}%
   {\par\medbreak}%
\newenvironment{proof}%
   {\medbreak\noindent{\it Proof:\space}}%
   {\par\noindent\vrule height 5pt width 5pt depth 0pt\smallbreak}%
\newcommand{\df}{\bf}
\let\sauvegardetiret=\-
\renewcommand{\-}[1]{\ifx#1-\penalty10000\hbox{-\relax}\penalty10000\else\sauvegardetiret#1\fi}
\newcommand{\tq}{\mathrel{:}}
\newcommand{\NN}{{\mathbf N}}

\newcommand{\cA}{{\cal A}}

\newcommand{\cC}{{\cal C}}

\newcommand{\cI}{{\cal I}}

\newcommand{\cL}{{\cal L}}

\newcommand{\Card}{\mathop{\rm Card}}
\newcommand{\topdim}{\mathop{\rm topdim}}
\newcommand{\Int}{\mathop{\rm Int}\nolimits}
\newcommand{\rk}{\mathop{\rm rk}}
\newcommand{\Th}{\mathop{\rm Th}}

\newcommand{\HM}[1]{(HM$_{#1}$)}

\title{Topological cell decomposition and dimension theory in $P$-minimal fields}
\author{
    Pablo Cubides-Kovacsics\footnote{Laboratoire Paul Painlev\'e,
    Universit\'e de Lille~1, CNRS U.M.R.~8524, 59655~Villeneuve d'Ascq
    Cedex, France. }
    \and
    Luck darni\`ere\footnote{LAREMA, Universit\'e d'Angers, 2~Bd
    Lavoisier,  49045~Angers Cedex~01, France.}
    \and
    Eva Leenknegt\footnote{Department of Mathematics, KULeuven,
    Celestijnenlaan~200B, 3001 Heverlee, Belgium.}
}

\begin{document}

\maketitle

\begin{abstract}
  This paper addresses some questions about dimension theory for
  $P$-minimal structures. We show that, for any definable set A, the
  dimension of $\overline{A}\backslash A$ is strictly smaller than the
  dimension of $A$ itself, and that $A$ has a decomposition into
  definable, pure-dimensional components. This is then used to show
  that the intersection of finitely many definable dense subsets of
  $A$ is still dense in $A$. As an application, we obtain that any
  definable function $f: D \subseteq K^m \to K^n$ is continuous on a dense,
  relatively open subset of its domain $D$, thereby answering a
  question that was originally posed by Haskell and Macpherson. 

   In order to obtain these results, we show that $P$-minimal structures
   admit a type of cell decomposition, using a topological
   notion of cells inspired by real algebraic geometry.
 \end{abstract}

\section{Introduction}

Inspired by the successes of $o$-minimality \cite{drie-1998} in real
algebraic geometry, Haskell and Macpherson \cite{hask-macp-1997} set
out to create a $p$-adic counterpart, a project which resulted in the
notion of $P$-minimality. One of their achievements was to build a
theory of dimension for definable sets which is in many ways similar
to the $o$-minimal case. Still, some questions remained open.

The theorem below  is one of the main results of this paper. It gives
a positive answer to one of the questions raised at the end of their
paper (Problem 7.5). We will assume $K$ to be a $P$-minimal expansion
of a $p$-adically closed field with value group $|K|$. When we say
definable, we mean definable (with parameters) in a $P$-minimal
structure.
\begin{theorem-nonum}[Quasi-Continuity] 
  Every definable function $f$ with domain $X\subseteq K^m$ and values in
  $K^n$ (resp. $|K|^n$) is continuous on a definable set $U$ which is
  dense and open in $X$, and $\dim (X\setminus U)<\dim U$.
\end{theorem-nonum}

Haskell and Macpherson already included a slightly weaker version of
the above result in Remark~5.5 of their paper \cite{hask-macp-1997},
under the additional assumption that $K$ has definable Skolem
functions. However, they only gave a sketch of the proof, leaving out
some details which turned out to be more subtle than expected. The
authors agreed with us that some statements in the original proof
required further clarification. 

One of the motivations for writing this paper was to remedy this, and
also to show that the assumption of Skolem functions could be removed.
This seemed worthwhile given that the result had already proven to be
a useful tool for  deducing other properties about the dimension of
definable sets in $P$-minimal structures. For example, in
\cite{kuij-leen-2015} one of the authors showed how the
Quasi-Continuity Theorem would imply the next result.  

\begin{theorem-nonum}[Small Boundaries] 
  Let $A$ be a non-empty definable subset of $K^m$.
  Then it holds that $\dim(\overline{A}\setminus A)<\dim A$. 
\end{theorem-nonum}
That both theorems are very much related is further illustrated by the
approach in this paper: we will first prove the Small Boundaries
Property, and use it to derive the Quasi-Continuity Property. The tool
used to prove these results is  a `topological cell decomposition',
which we consider to be the second main result of this paper.

\begin{theorem-nonum}[Topological Cell Decomposition]
  For every definable function $f$ from $X\subseteq K^m$ to $K^n$ (resp.
  $|K|^n$) there exists a good t\--cell decomposition $\cA$ of $X$, such
  that for every $A\in\cA$, the restriction $f_{|A}$ of $f$ to $A$ is
  continuous.
\end{theorem-nonum}
The notions of `t\--cell' and `good t-cell decomposition' were
originally introduced by Mathews, whose paper \cite{math-1995} has
been  a major source of inspiration for us. They are analogous to a
classical notion of cells coming from real algebraic geometry (see for
example the definition of cells in \cite{boch-cost-roy-1987}).  Exact
definitions will be given in the next section. 
\\\\
By now, there exist many cell decomposition results for P-minimal
structures, which can be quite different in flavour, depending on
their aims and intended level of generality. Historically, the most
influential result is probably Denef's cell decompositon for
semi-algebraic sets \cite{dene-1986} (which in turn was inspired by
Cohen's work \cite{cohe-1969}). This has inspired adaptations to the
sub-analytic context by Cluckers \cite{cluc-2004}, and to multi-sorted
semi-algebraic structures by Pas \cite{pas-1990}. Results like
\cite{cubi-leen-2015, mour-2009, darn-halu-2015-tmp} give
generalizations of Denef-style cell decomposition. Note that full
generality is hard to achieve: whereas \cite{cubi-leen-2015} works for
all $p$-minimal structures without restriction, it is somewhat weaker
than these more specialized results. On the other hand,
\cite{mour-2009, darn-halu-2015-tmp} are closer to the results cited
above, but require some restrictions on the class of $P$\--minimal
fields under consideration. A somewhat different result is the
Cluckers-Loeser cell decomposition \cite{cluc-loes-2007} for
$b$-minimal structures.

Each of these decompositions has its own strengths and weaknesses. The
topological cell decomposition proposed here seems to be the best for
our purposes, since it is powerful enough to fill the remaining
lacunas in the dimension theory of definable sets over $P$\--minimal
fields, without restriction.

\
  
The rest of this paper will be organized as follows. In
section~\ref{se:notation}, we recall some definitions and known
results, and we set the notation for the remaining sections. In
section~\ref{se:cell-dec}, we will prove the t\--cell decomposition
theorem (Theorem~\ref{th:M-cell-prep}) and deduce the Small Boundaries
Property (Theorem~\ref{th:dim-boundary}) as a corollary. Finally, in
section~\ref{se:pure-dim}, we prove the Quasi-continuity Property
(Theorem~\ref{th:dense-cont}). The key ingredient of this proof is the
following result (see Theorem~\ref{th:dense-int}), which is also
interesting in its own right.

\begin{theorem-nonum}
  Let $A_1,\dots,A_r\subseteq A$ be a family of definable subsets of $K^m$. If the
  union of the $A_k$'s has non empty interior in $A$ then at least one
  of them has non empty interior in $A$. 
\end{theorem-nonum}

Note that the above statement shows that, if $B_1,\dots,B_r$ are definable
subsets which are dense in $A$, then their intersection $B_1\cap\cdots\cap B_r$
will also be dense in $A$. Indeed, a definable subset is dense in $A$
if and only if its complement in $A$ has empty interior inside $A$.

\paragraph{Acknowledgement} 
The authors would like to thank Raf Cluckers for encouraging this
collaboration and for helpful discussions. The research leading to
these results has received funding from the European Research Council,
ERC Grant nr. 615722, MOTMELSUM, 2014--2019. During the preparation of
this paper, the third author was a Postdoctoral Fellow of the Fund
for Scientific Research - Flanders (FWO)

\section{Notation and prerequisites}
\label{se:notation}

Let $K$ be a $p$-adically closed field, i.e., elementarily equivalent
to a $p$-adic field,  and $K^*=K\setminus\{0\}$. We use multiplicative
notation for the $p$\--valuation, which we then denote by $|\,.\,|$ so
$|ab|=|a||b|$, $|a+b|\leqslant\max |a|,|b|$, and so on\footnote{
Compared with additive notation this reverses the order:
$|a|\leqslant|b| \Leftrightarrow v(b)\leqslant v(a)$.}\!. For every set $X\subseteq K$ we
will use the notation $|X|$ for the image of $X$ by the valuation.  
A natural way to  extend the valuation  to $K^m$ is by putting 
\[\|(x_1,\dots,x_m)\|:=\max_{i\leqslant m}\{|x_i|\}.\] 
This induces a topology, with balls 
\[B(x,\rho):=\{y\in K^m\tq\|x-y\|<\rho\}\]
as basic open sets, where $x\in K^m$ and $\rho\in |K^*|$. For every $X\subseteq
K^m$, write $\overline{X}$ for the closure of $X$ and $\Int X$ for the
interior of $X$ (inside $K^m$). The relative interior of a subset $A$
of $X$ inside $X$, that is $X\setminus\overline{X\setminus A}$, is denoted $\Int_X A$.

Let us now recall the definition of $P$-minimality:
\begin{definition}
  Let $\mathcal{L}$ be a language extending the ring language
  $\mathcal{L}_{\text{ring}}$. A structure $(K, \mathcal{L})$ is said
  to be $P$-minimal if, for every structure $(K', \mathcal{L})$
  elementarily equivalent to $(K,\mathcal{L})$, the
  $\mathcal{L}$-definable subsets of $K'$ are semi-algebraic.
\end{definition}
In this paper, we always work in a $P$-minimal structure
$(K,\cL)$. Abusing notation, we simply denote it as $K$.
The word definable means definable using parameters in $K$.
A set $S \subseteq K^m\times|K|^n$ is said to be definable if the
related set $\{(x,y) \in K^m\times K^n \tq (x,|y_1|, \ldots |y_n|)\in S\}$ is
definable. 

A function $f$ from $X\subseteq K^m$ to $K^n$ (or to $|K|^n$) is definable if
its graph is a definable set. For every such function, let $\cC(f)$
denote the set
\begin{displaymath} 
  \cC(f):=\big\{a\in X\tq f\mbox{ is continuous on a neighbourhood of $a$
  in }X\big\}.
\end{displaymath}
It is easy to see that this is a definable set.

We will use the following notation for the fibers of a set. For any
set $S\subseteq K^{m}$, the subsets $I = \{i_1, \ldots, i_r\}$ of $\{1, \ldots,
m\}$ induce  projections $\pi_I:K^m\to K^r$ (onto the coordinates listed in
$I$) . Given an element $y\in K^r$, the fiber $X_{y,I}$ denotes the set
$\pi_I^{-1}(y)\cap X$. In most cases, we will drop the sub-index $I$ and
simply write $X_y$ instead of $X_{y,I}$ when the projection $\pi_I$ is
clear from the context. In particular, when $S\subseteq K^{m+n}$ and $x\in K^m$,
we write $S_x$ for the fiber with respect to the projection onto the
first $m$ coordinates.  
\\\\
One can define a strict order on the set of non-empty definable
subsets of $K^m$, by putting
\begin{center}
  $B\ll A$ \quad $\Leftrightarrow$ \quad $B\subseteq A$ and $B$ lacks interior in $A$.
\end{center} 
The rank of $A$ for this order is denoted $D(A)$. It is defined by
induction: $D(A)\geqslant 0$ for every non-empty set $A$, and
$D(A)\geqslant d+1$ if there is a non-empty definable set $B\ll A$ such
that $D(B)\geqslant d$. Then $D(A)=d$ if $D(A)\geqslant d$ and
$D(A)\ngeqslant d+1$. By convention $D(\emptyset)=-\infty$. 
\\\\
The notion of dimension used by Haskell and Macpherson in
\cite{hask-macp-1997} (which they denoted as $\topdim A$) is defined
as follows:
\begin{definition} \label{def:dimension}
  The \textbf{dimension} of a set $A \subset K^m$ (denoted as $\dim A$) is
  the maximal integer $r$ for which there exists a subset $I$ of
  $\{1,\dots,m\}$ such that $\pi_I^m(A)$ has non-empty interior in $K^r$,
  where $\pi_I^m:K^m\to K^r$ is defined by
  \begin{displaymath}
    \pi_I^m:(x_1,\dots,x_m)\mapsto (x_{i_1},\dots,x_{i_r}) 
  \end{displaymath}
  with $i_1<\cdots<i_r$ an enumeration of $I$.
\end{definition}
We will omit the super-index $m$ in $\pi_I^m$ when it is clear from the
context, and put $\dim\emptyset=-\infty$. Given a set $S\subseteq K^{m+1}$,
$\pi^{m+1}_{\{1,\dots,m\}}(S)$ is simply denoted $\widehat{S}$. 

Note that by $P$-minimality, if $A \subseteq K^m$ is a definable set
and $\dim A =0$, then $A$ is a finite set. Also, $\dim A = m$ if and
only if $A$ has non-empty interior.  
\\\\
Let us now recall some of the properties of this dimension that were
already proven by Haskell and Macpherson in \cite{hask-macp-1997}:
{\sl
  \begin{description}
    \item[\HM {\bf 1}]
      Given definable sets $A_1,\dots,A_r\subseteq K^m$, it holds that $\dim
      A_1\cup\cdots\cup A_r=\max(\dim A_1,\dots,\dim A_r)$. (Theorem~3.2)
    \item[\HM {\bf 2}] 
      For every definable function $f:X\subseteq K^m\to |K|$, $\dim X\setminus\cC(f)<m$.
      (Theorem~3.3 and Remark~3.4 (rephrased))
    \item[\HM {\bf 3}]
      For every definable function $f:X\subseteq K^m\to K$, $\dim X\setminus\cC(f)<m$.
      (Theorem~5.4)
      \end{description}}

\noindent
Recall that a complete theory $T$ satisfies the  Exchange Principle if
the model-theoretic algebraic closure for $T$ does so. In every model
of a theory satisfying the Exchange Principle, there is a well-behaved
notion of dimension for definable sets, which is called model
theoretic rank. Haskell and Macpherson showed that
{\sl
  \begin{description}
    \item[\HM {\bf 4}]
      The model-theoretic algebraic closure for $\Th(K)$ satisfies the
      Exchange Principle. (Corollary~6.2)
    \item[\HM {\bf 5}] 
      For every definable $X\subseteq K^m$, $\dim X$ coincides with the
      model-theoretic $\rk X$. (Theorem~6.3)
  \end{description}
}

The following Additivity Property (Lemma~\ref{le:additivity} below) is
known to hold for the model theoretic rank $rk$, in theories
satisfying the exchange principle. For a proof, see Lemma~9.4 in
\cite{math-1995}. Hence, theorems \HM 4 and \HM 5 imply that $\dim$
also satisfies the Additivity Property. This fact was not explicitly
stated by Haskell and Macpherson in \cite{hask-macp-1997}, and seems
to have been somewhat overlooked until now. It plays a crucial role in
our proof of Theorem~\ref{th:M-cell-prep}, hence in all our paper.

\begin{lemma}[Additivity Property]\label{le:additivity}
  Let $S \subseteq K^{m+n}$ be a definable set. For $d \in \{ -\infty, 0,1, \ldots,
  n\}$, write $S(d)$ for the set 
  \[ S(d) := \{ a \in K^m \tq \dim S_a = d \}.\] 
  Then  $S(d)$ is definable and 
  \[ \dim \bigcup_{a \in S(d)} S_a= \dim(S(d)) + d.\]
\end{lemma}

\noindent 
Combining this with the first point \HM1, it follows easily that
$\dim$ is a dimension function in the sense of van den Dries
\cite{drie-1989}.
\\\\
Haskell and Macpherson also proved that $P$-minimal structures are
\textbf{model-theoretically bounded} (also known as ``algebraically
bounded'' or also that ``they eliminate the $\exists^\infty$ quantifier''),
\textit{i.e.}, for every definable set $S\subseteq K^{m+1}$ such that all the
fibers of the projection of $S$ onto $K^m$ are finite, there exists an
integer $N\geqslant 1$ such that all of them have cardinality
$\leqslant N$. 
\\\\
While it is not known whether general $P$-minimal structures admit
definable Skolem functions, we do have the following weaker version
for coordinate projections with finite fibers.

\begin{lemma}
  Let $S \subseteq K^{m+1}$ be a definable set.
  Assume that all fibers $S_x$ with respect to the projection onto the
  first $m$ coordinates are finite. Then there exists a definable
  function $\sigma: \widehat{S} \to K^{m+1}$  such that $\sigma(x)\in S$ for every
  $x \in \widehat{S}$.
\end{lemma}

\begin{proof}
In Lemma 7.1 of \cite{dene-1984}, Denef shows that this is true on the
condition that the fibers are not only finite, but uniformly bounded.
(The original lemma was stated for semi-algebraic sets, but the same
proof holds for general $P$-minimal structures.) Since uniformity is
guaranteed by model-theoretic boundedness, the lemma follows.
\end{proof}
From this it follows by an easy induction that
\begin{corollary}[Definable Finite Choice] \label{cor:choice}
  Let $f: X \subseteq K^m \to K^n$ be a new definable function. Assume that for
  every $y \in f(X)$, $f^{-1}(y)$ is finite. Then there exists a
  definable function $\sigma: f(X) \to X$, such that 
  \[\big(\sigma\circ f(x), f(x)\big) \in \text{Graph}(f)\] 
  for all $x \in X$.
\end{corollary}

Using the coordinate projections from Definition \ref{def:dimension},
we will now give a definition of t\--cells and t\--cell decomposition:
\begin{definition}
  A set $C \subseteq K^m$ is a {\df topological cell} (or {\df t\--cell} for
  short) if there exists some (non unique) $I\subseteq\{1,\dots,n\}$ such that
  $\pi_I^m$ induces a homeomorphism from $C$ to a non-empty open set.
\end{definition}
In particular, every non-empty open subset of $K^m$ is a t\--cell, and
the only finite t\--cells in $K^m$ are the points. For any definable
set $X\subseteq K^m$, a {\df t\--cell decomposition} is a partition $\cA$ of
$X$ in finitely many t\--cells. We say that the  t\--cell
decomposition is {\df good}, if moreover each t\--cell in $\cA$ is
either open in $X$ or lacks interior in $X$.

\section{Topological cell decomposition}
\label{se:cell-dec}

Recall that $K$ is a $P$-minimal expansion of a $p$-adically closed
field. We will first show that every set definable in such a structure
admits a decomposition in t\--cells:

\begin{lemma}\label{le:M-cell-dec}
  Every definable set $X\subseteq K^m$ has a good t\--cell decomposition. 
\end{lemma}

\begin{proof}
Put $d=\dim X$ and let $e=e(X)$ be the number of subsets $I$ of $\{1,\dots,m\}$
for which $\pi_I(X)$ has non-empty interior in $K^d$. The proof goes by
induction on pairs $(d,e)$ (in lexicographic order). The result is
obvious for $d\leqslant 0$ so let us assume that $1\leqslant d$, and that
the result is proved for smaller pairs.

Let $I\subseteq\{1,\dots,m\}$ be such that $\pi_I(X)$ has non-empty interior in $K^d$.
For every $y$ in $\Int\pi_I(X)$, we write $X_y$ for the fiber with
respect to the projection $\pi_I$. For every integer $i\geqslant 1$ let
$W_i$ be the set  \[W_i := \{y\in\Int\pi_I(X) \tq \Card X_y = i\}.\] By
model-theoretic boundedness, there is an integer $N\geqslant 1$ such
that $W_i$ is empty for every $i> N$. We let $\cI$ denote the set of
indices $i$ for which $W_i$ has non-empty interior in $K^d$. 

For each $i\in\cI$, Definable Finite Choice (Corollary~\ref{cor:choice})
induces a definable function \[\sigma_i: =(\sigma_{i,1},\dots,\sigma_{i,i}): \Int W_i \to
K^{mi},\] such that $X_y=\{\sigma_{i,j}(y)\}_j$ for every {$y\in \Int  W_i$}.
Put $V_i: =\cC(\sigma_i)$, and $C_{i,j}:=\sigma_{i,j}(V_i)$. Notice that
$C_{i,j}$ is a t\--cell for every $i\in\cI$ and $j\leqslant i$. Indeed,
the restrictions of $\pi_I$ and $\sigma_{i,j}$ are reciprocal homeomorphisms
between $C_{i,j}$ and the open set $V_i$. We show that each $C_{i,j}$
is open in $X$.

Fix $i \in \mathcal{I}$ and $j \leqslant i$. Let $x_0$ be an element of
$C_{i,j}$ and $y_0=\pi_I(x_0)$, so that $x_0 = \sigma_{i,j}(y_0)$. By
construction, $\bigcup_k C_{i,k}=\pi_I^{-1}(V_i)\cap X$ is open in $X$ (because
$V_i=\cC(\sigma_i)$ is open in $\Int W_i$, hence open in $K^d$). So there is
$\rho\in|K^\times|$ such that $B(x_0,\rho)\cap X$ is contained in $\bigcup_k C_{i,k}$. Let
$\varepsilon$ be defined as 
\begin{displaymath}
  \varepsilon:=\min_{k\neq j}\|\sigma_{i,k}(y_0)-\sigma_{i,j}(y_0)\| = \min_{k\neq
  j}\|\sigma_{i,k}(y_0)-x_0\|
\end{displaymath}
Because $\sigma_i$ is continuous on the open set $V_i$, there exists $\delta$
such that 
\begin{displaymath}
  B(y_0, \delta) \subseteq \sigma_i^{-1}(B(\sigma_i(y_0),\rho)) \subseteq V_i, 
\end{displaymath}
and such that for all $y \in B(y_0, \delta)$, we have that
\begin{displaymath}
  \|\sigma_i(y) - \sigma_i(y_0)\|< \varepsilon. 
\end{displaymath}
Making $\delta$ smaller if necessary, we may assume that $\delta < \min\{\varepsilon, \rho\}$.
We will show  that $B(x_0, \delta) \cap X \subseteq C_{i,j}$. Let  $x$ be in $B(x_0,
\delta) \cap X$, and put $y:= \pi_I(x)$. Since $\delta < \rho$, we know that there must
exist $k$ such that $x = \sigma_{i,k}(y)$. Assume that $k \neq j$. Since $\delta <
\varepsilon$, we now have that 
\begin{eqnarray*}
\| \sigma_{i,k}(y) - \sigma_{i,k}(y_0)\| & \leqslant & \| \sigma_i(y) - \sigma_i(y_0)\|\\
& < & \varepsilon \\
& \leqslant & \| \sigma_{i,j}(y_0) - \sigma_{i,k}(y_0)\| \\
& = & \| \sigma_{i,k}(y) - \sigma_{i,j}(y_0) \|\\
& = & \| \sigma_{i,k}(y) - x_0\|,
\end{eqnarray*}
but this means that $\sigma_{i,k}(y) \not \in B(x_0, \delta)$, and hence we can
conclude that $x = \sigma_{i,j}(y) \in C_{i,j}$.

Given that each $C_{i,j}$ is a t\--cell which is open in $X$, it remains
to show the result for $Z:=X\setminus (\bigcup_{i\in\cI, j\leqslant i} C_{i,j})$. We
will check that $\pi_I(Z)$ has empty interior (in $K^d$), or
equivalently that $\dim\pi_I(Z)< d$. 

Note that $\pi_I(Z)$ is a disjoint union $A_1 \sqcup A_2 \sqcup A_3$, where
$A_1:= \pi_I(X)\setminus\Int\pi_I(X)$, $A_2 := \Int\pi_I(X)\setminus\bigcup_{i\leqslant N}W_i$,
and $A_3$ is the set
\begin{displaymath}
  A_3:=  \bigg(\bigcup_{i\in\cI} \big(W_i\setminus\Int W_i\big)
       \cup \big(\Int W_i\setminus V_i\big)\bigg)
       \cup \bigcup_{i\notin\cI}W_i.
\end{displaymath}
By \HM 1 it suffices to check that each of these parts has dimension
$<d$. Clearly $A_1$ has empty interior, hence dimension $<d$. For
every $y$ in $A_2$, the fiber $X_y$ is infinite, hence $A_2$ must have
dimension $<d$  by the Additivity Property. 

Next, we need to check that $A_3$ also has dimension smaller than $d$.
By \HM 1, it is sufficient to do this for each  part separately. The
set $W_i\setminus \Int W_i$ has empty interior for every $i\in\cI$, and hence
dimension  $<d$.  For $i\in \cI$, $\Int W_i\setminus V_i$ has dimension $<d$ by
\HM 3. And finally, $W_i$ has empty interior for every $i\notin\cI$ by
definition of $\cI$, hence dimension $<d$. So $\dim \pi_I(Z)<d$ by \HM
1, hence $\pi_I(Z)$ has empty interior. 

{A fortiori,} the same holds for $\pi_I(Z_1)$ and $\pi_I(Z_2)$ where
$Z_1:=\Int_X Z$ and $Z_2:=Z\setminus\Int_X Z$. This implies that, for each
$k\in\{1,2\}$, either $\dim Z_k< d$, or $\dim Z_k=d$ and $e(Z_k)<e$.
Hence, the induction hypothesis applies to each $Z_k$ separately and
gives a good partition $(D_{k,l})_{l\leqslant l_k}$ of $Z_k$. Since
$Z_1$ is open in $X$ and $Z_2$ has empty interior in $X$, the sets
$D_{k,l}$ will also be either open in $X$, or have empty interior in
$X$. It follows that the family of consisting of the t\--cells
$C_{i,j}$ and $D_{k,l}$ forms a good t\--cell decomposition of $X$.

\end{proof}
We will now show that this decomposition can be chosen in such a way
as to ensure that definable functions are piecewise continuous, which
is one of the main theorems of this paper.

\begin{theorem}[Topological Cell Decomposition]\label{th:M-cell-prep}
  For every definable function $f$ from $X\subseteq K^m$ to $K^n$ (or to
  $|K|^n$) there exists a good t\--cell decomposition $\cC$ of $X$, such
  that for every $C\in\cC$ the restriction $f_{|C}$ of $f$ to $C$ is
  continuous. 
\end{theorem}

\begin{proof}
We prove the result for functions $f:X\subseteq K^m\to K^n$, by induction on
pairs $(m,d)$ where $d=\dim X$. Our claim is obviously true if $m=0$
or $d\leqslant 0$, so let us assume that $1\leqslant d\leqslant m$ and
that  the theorem holds for smaller pairs.

Note that it suffices to prove the result  for each coordinate
function $f_i$ of $f:= (f_1,\dots,f_n)$ separately. Indeed,  suppose the
theorem is true for the functions $f_i: X \to K$. This means that, for
each $1\leqslant i\leqslant n$, there exists a good t\--cell
decomposition $\cC_i$ of $X$ adapted to $f_i$. It is then easy, by
means of Lemma~\ref{le:M-cell-dec}, to build a common, finer good
t\--cell decomposition of $X$ having the required property
simultaneously for each $f_i$, and hence for $f$. Thus, we may as well
assume that $n=1$. 

Consider the set $X\setminus\Int\cC(f)$, which can be partitioned as $ A_1 \sqcup
A_2$, where $A_1 := X\setminus\cC(f)$ and $A_2:=\cC(f)\setminus\Int\cC(f)$.  It
follows from \HM 3 that $\dim A_1 <m$. Also, $\dim A_2<m$ since it has
empty interior (inside $K^m$),  and therefore the union,
$X\setminus\Int\cC(f)$, has dimension $<m$ by \HM 1. Hence, by throwing away
$\Int\cC(f)$ if necessary (which is a definable open set contained in
$X$, hence a t\--cell open in $X$ if non-empty), we may assume that
$\dim X<m$. 

Using Lemma~\ref{le:M-cell-dec}, one can obtain a good t\--cell
decomposition $(X_j)_{j\in J}$ of $X$. For each $j\in J$, we get a subset
$I_j$ of $\{1,\dots,m\}$, an open set $U_j \subseteq K^{d_j}$ (with $d_j=\dim
X_j<m$), and a definable map $\sigma_j:U_j\to X_j$. These maps $\sigma_j$ can be
chosen in such a way that $\sigma_j$ and the restriction of $\pi_{I_j}$ to
$X_j$ are reciprocal homeomorphisms. Now apply the induction
hypothesis to each of the functions $f\circ\sigma_j$ to get a good t\--cell
decomposition $\cC_j$ of $X_j$. Putting $\cC=\bigcup_{j\in J}\cC_j$ then
gives the conclusion for $f$. 

The proof  for functions $f:X\subseteq K^m\to|K|^n$ is similar, the main
difference being that one needs to use \HM 2 instead of \HM 3. 
\end{proof}

\begin{remark}\label{re:M-cell-dens}
  With the notation of Theorem~\ref{th:M-cell-prep}, let $U$ be the
  union of the cells in $\cC$ which are open in $X$. Clearly
  $U\subseteq\cC(f)$ and $X\setminus U$ is the union of the other cells in $\cC$, each
  of which lacks interior in $X$. To conclude that $\cC(f)$ is dense
  in $X$, it remains to check that this union still has empty interior
  in $X$. This will be the subject of section~\ref{se:pure-dim}.
\end{remark}

The Topological Cell Decomposition Theorem is a strict analagon of the
Cell Decomposition Property (CDP) considered by Mathews in the more
general context of  $t$-minimal structures. In his paper, Mathews
showed that the CDP holds in general for such structures, if a number
of rather restrictive conditions hold (e.g., he assumes that the
theory of a structure has quantifier elimination), see Theorem 7.1 in
\cite{math-1995}. Because of these restrictions, we could not simply
refer to this general setting for a proof of the CDP for $P$-minimal
structures.

Further results from Mathews' paper justify why proving Theorem
\ref{th:M-cell-prep} is worth the effort.  In Theorem~8.8 of
\cite{math-1995} he shows that, if the CDP and the Exchange Principle
are satisfied for a t\--minimal structure with a Hausdorff topology,
then several classical notions of ranks and dimensions, including $D$
and $\dim$, coincide for its definable sets. Because of
Theorem~\ref{th:M-cell-prep} and \HM 4, we can now apply the
observation from Theorem 8.8 to $P$-minimal fields, to get that

\begin{corollary}\label{co:rk-dim-D}
  For every definable set $A\subseteq K^m$, $\dim A=D(A)$. 
\end{corollary}
The Small Boundaries Property then follows easily. 

\begin{theorem}[Small Boundaries Property]\label{th:dim-boundary}
  For every definable set $A\subseteq K^m$, one has that $\dim(\overline{A}\setminus
  A)<\dim A$. 
\end{theorem}

\begin{proof}
First note that $D(\overline{A}\setminus A)<D(\overline{A})$, since
$\overline{A}\setminus A$ has empty interior in $\overline{A}$. This means
that $\dim(\overline{A}\setminus A)<\dim\overline{A}$ by
Corollary~\ref{co:rk-dim-D}. Applying \HM 1, we get that $\dim
\overline{A}=\dim A$, and therefore  $\dim(\overline{A}\setminus A)<\dim A$. 
\end{proof}

\section{Relative interior and pure components}
\label{se:pure-dim}

Given a definable set $A\subseteq K^m$ and $x\in K^m$, let $\dim(A,x)$ denote
the smallest $k\in\NN\cup\{-\infty\}$ for which there exists a ball $B\subseteq K^m$
centered at $a$, such that $\dim A\cap B=k$ (see for
example~\cite{boch-cost-roy-1987}).  Note that $\dim(A,x)=-\infty$ if and
only if $x\notin \overline{A}$. We call this the {\df local dimension} of
$A$ at $x$. $A$ is said to be {\df pure dimensional} if it has the
same local dimension at every point $x\in A$. 

\begin{claim}\label{cl:pure-dim-basic}
  Let $S\subseteq K^m$ be a definable set of pure dimension $d$.
  \begin{enumerate}
    \item\label{it:pur-dim-dense}
      Every definable set dense in $\overline{S}$ has pure dimension
      $d$. 
    \item\label{it:pur-dim-int}
      For every definable set $Z\subseteq S$, $Z$ has empty interior in $S$ if
      and only if $\dim Z<\dim S$.
  \end{enumerate}
\end{claim}

\begin{proof}
Let $X\subseteq S$ be a definable set dense in $S$. Consider a ball $B$ with
center $x\in X$. Then $B\cap S$ is non-empty, and therefore we have that
$\dim B\cap S=d$. Moreover, it is easy to see that $B\cap X$ is dense in $B\cap
S$, which implies that $\dim B\cap X=d$ as well, by the Small Boundaries
Property and \HM 1. This proves the first part. 

Let us now prove the second point. If $Z$ has empty interior in $S$,
this means that $S\setminus Z$ is dense in $S$, and hence $Z$ is contained in
$\overline{(S\setminus Z)}\setminus (S\setminus Z)$. But then $\dim Z<\dim (S\setminus Z)$ by the
Small Boundaries Property, and therefore $\dim Z<\dim S$. Conversely,
if $Z$ has non-empty interior inside $S$, there exists a ball $B$
centered at a point $z\in Z$ such that $B\cap S\subseteq Z$. By the purity of $S$,
$\dim B\cap S=d$,  and hence $\dim Z\geqslant d$. Since $Z\subseteq S$, this
implies that $\dim Z=d$.  
\end{proof}

\noindent For every positive integer $k$, we put
\[ \Delta_k(A):= \{ a \in A \mid \dim(A,a) =k \},\]
and we write $C_k(A)$ for the topological closure of $\Delta_k(A)$ inside
$A$. It is easy to see that $\Delta_k(A)$ is pure dimensional, and of
dimension $k$ if the set is non-empty. By part~\ref{it:pur-dim-dense}
of Claim~\ref{cl:pure-dim-basic}, the same holds for $C_k(A)$ .
Moreover, since $C_k(A)$ is closed in $A$, one can check that it is
actually the largest definable subset of $A$ with pure dimension $k$
(if it is non-empty). For this reason, we call the sets $C_k(A)$ the
{\df pure dimensional components} of $A$. 

\begin{remark}\label{re:Delta-closure}
  If $\dim(A,x)<k$ for some $x\in A$, then there exists a ball $B$
  centered at $x$ for which $\dim B\cap A<k$. Such a ball must be
  disjoint from $C_k(A)$, because $C_k(A)$ either has pure dimension
  $k$ or is empty. But then $C_k(A)$ is disjoint from every $\Delta_l(A)$
  with $l<k$, which means that it must be contained in the union of
  the $\Delta_l(A)$ with $l\geqslant k$. 
\end{remark}

\begin{lemma}\label{le:dim-CkCl}
  For every definable set $A\subseteq K^m$ and every $k$, one has that
  \begin{displaymath}
    \dim \left(C_k(A)\cap \bigcup_{l\neq k} C_l(A)\right) < k.
  \end{displaymath}
\end{lemma}

\begin{proof}
By \HM 1, it suffices to check that $\dim C_k(A)\cap C_l(A)<k$ for every
$l\neq k$.  This is obvious when $l<k$,  since in these cases $C_l(A)$
already has dimension $l$ or is empty. Hence we may assume that $l>k$.
Using Remark~\ref{re:Delta-closure}, one gets that
\begin{displaymath}
  C_k(A)\cap C_l(A)\subseteq C_k(A)\cap\bigcup_{i>k}\Delta_i(A) = \bigcup_{i>k}C_k(A)\cap \Delta_i(A).
\end{displaymath}
Using \HM 1 again, it now remains to check that $\dim C_k(A)\cap\Delta_i(A)<k$
whenever $i>k$. But since $\Delta_i(A)$ is disjoint from $\Delta_k(A)$, we find
that $C_k(A)\cap\Delta_i(A)\subseteq C_k(A)\setminus \Delta_k(A)$. This concludes the proof because
of the Small Boundaries Property.
\end{proof}

\begin{lemma}\label{le:Ck-int-vide}
  Let $Z\subseteq A\subseteq K^m$ be definable sets. Then $Z$ has empty interior
  inside $A$ if and only if $\dim Z\cap C_k(A)<k$ for every $k$.
\end{lemma}

\begin{proof}
For every $k$, we will consider the set  \[D_k(A)=A\setminus\bigcup_{l\neq k}C_l(A).\]
Clearly, this set is open in $A$ and contained in $\Delta_k(A)$. We claim
that $D_k(A)$ is also dense in $C_k(A)$. Indeed, $C_k(A)$ is either
empty or has pure dimension $k$. The first case is obvious, so assume
that $C_k(A)$ has pure dimension $k$. By part~\ref{it:pur-dim-int} of
Claim~\ref{cl:pure-dim-basic}, it suffices to check that $C_k(A)\setminus
D_k(A)$ has dimension $<k$. But this follows from
Lemma~\ref{le:dim-CkCl}, so our claim holds.

If $Z$ has non-empty interior in $A$,  there exists $z\in Z$ and
$r\in|K^\times|$, such that $B(z,r)\cap A\subseteq Z$. If we put $k:=\dim(A,z)$, then
$z\in \Delta_k(A)$. Since $D_k(A)$ is dense in $\Delta_k(A)$, the set $D_k(A)\cap
B(z,r)$ is non-empty. Pick a point $z'$ in this intersection. Because
$D_k(A)$ is open in $A$, there exists $r'\in|K^\times|$ such that $B(z',r')\cap
A\subseteq D_k(A)$ and $r' \leqslant r$. But then \[ B(z',r')\cap D_k(A) \subseteq
B(z',r)\cap A = B(z,r)\cap A\subseteq Z, \] and $B(z',r')\cap D_k(A)$ is non-empty
since it contains $z'$. This shows that $Z\cap D_k(A)$ has non-empty
interior inside $D_k(A)$.  Since $D_k(A)$ is open in $A$ (and hence in
$C_k(A)$),  $Z\cap D_k(A)$ has non-empty interior inside $ C_k(A)$ as
well. Because $C_k(A)$ is pure dimensional, part~\ref{it:pur-dim-int}
of Claim~\ref{cl:pure-dim-basic} implies that $\dim (Z\cap C_k(A))=k$. 

Conversely, assume that $\dim (Z\cap C_k(A))=k$ for some $k$. By the
Small Boundaries Property, one has that $\dim (C_k(A)\setminus D_k(A))<k$.
From this, we can deduce that $\dim (Z\cap D_k(A))=k$, using \HM 1. The
purity of $D_k(A)$ and part~\ref{it:pur-dim-int} of
Claim~\ref{cl:pure-dim-basic} then imply that $Z\cap D_k(A)$ has
non-empty interior in $D_k(A)$, and hence in $A$ (since $D_k(A)$ is
open in $A$).  A fortiori, $Z$ itself has non-empty interior in $A$.
\end{proof}

We can now prove the results which were the aim of this section. 

\begin{theorem}\label{th:dense-int}
  Let $A_1,\dots,A_r\subseteq A$ be a finite family of definable subsets of $K^m$.
  If their union has non empty interior in $A$ then at least one of
  them has non empty interior in $A$. In particular, a piece $A_i$ has
  non-empty interior in $A$ if $\dim A_i\cap C_k(A)=k$ for some $k$. 
\end{theorem}

\begin{proof}
If $Z:=A_1\cup\cdots\cup A_r$ has non-empty interior in $A$, then $\dim (Z\cap
C_k(A))=k$ for some $k$ by Lemma~\ref{le:Ck-int-vide}. Then by \HM 1,
$\dim (A_i\cap C_k(A))=k$ for some $i$ and some $k$, and thus $A_i$ has
non-empty interior in $A$ by Lemma~\ref{le:Ck-int-vide}.
\end{proof}

\begin{theorem}\label{th:dense-cont}
  Every definable function $f$ from $X\subseteq K^m$ to $K^n$ (resp. $|K|^n$)
  is continuous on a definable set $U$ which is dense and open in $X$, and
  $\dim (X\setminus U)<\dim X$.
\end{theorem}

\begin{proof}
The existence of $U$, dense and open in $X$ on which $f$ is
continuous, follows from Theorems~\ref{th:M-cell-prep} and
\ref{th:dense-int} by Remark~\ref{re:M-cell-dens}. That $\dim (X\setminus
U)<\dim X$ then follows from the Small Boundaries Property.
\end{proof}

%%fakesection

\bibliographystyle{alpha}
% \bibliography{biblio}

\end{document}